\documentclass[12pt,reqno]{amsart}
\usepackage{amsmath,amssymb,amsthm,bbm,mathtools,calc,verbatim,enumitem,tikz,url,mathrsfs,cite,fullpage,hyperref}
\usepackage{float}
\usepackage{subcaption}

\usepackage{comment}

\newtheorem{theorem}{Theorem}[section]
\newtheorem{lemma}[theorem]{Lemma}

\theoremstyle{definition}

\newtheorem*{RBalg}{The Multicolour Book Algorithm}
\theoremstyle{remark}
\newtheorem{remark}[theorem]{Remark}

\newcommand\N{\mathbb{N}}
\newcommand\R{\mathbb{R}}
\newcommand\Z{\mathbb{Z}}

\newcommand\cB{\mathcal{B}}

\def\Pr{\mathbb{P}}

\newcommand\Ex{\mathbb{E}}
\newcommand\id{\hbox{$1\mkern-6.5mu1$}}

\newcommand\eps{\varepsilon}

\renewcommand{\le}{\leqslant}
\renewcommand{\ge}{\geqslant}
\renewcommand{\to}{\rightarrow}

\def\eps{\varepsilon}

	\def\R{\mathbb{R}}

	\def\Z{\mathbb{Z}}
	\def\N{\mathbb{N}}
	
	\def\1{\mathbbm{1}}

	\def\<{\langle}
	\def\>{\rangle}	

	\def\cB{\mathcal{B}}

\begin{document}

\title{Upper bounds for multicolour Ramsey numbers}

\author{Paul Balister \and B\'ela Bollob\'as \and Marcelo Campos \and Simon Griffiths \and Eoin Hurley\and Robert Morris \and Julian Sahasrabudhe \and Marius Tiba}

\address{Mathematical Institute, University of Oxford, Radcliffe Observatory Quarter, Woodstock Road, Oxford, OX2 6GG, UK}
\email{Paul.Balister@maths.ox.ac.uk}

\address{Department of Pure Mathematics and Mathematical Statistics, Wilberforce Road, Cambridge, CB3 0WA, UK, and Department of Mathematical Sciences, University of Memphis, Memphis, TN 38152, USA}
\email{bb12@cam.ac.uk}

\address{IMPA, Estrada Dona Castorina 110, Jardim Bot\^anico, Rio de Janeiro, 22460-320, Brasil}\email{marcelo.campos@impa.br}

\address{Departamento de Matem\'atica, PUC-Rio, Rua Marqu\^{e}s de S\~{a}o Vicente 225, G\'avea, 22451-900 Rio de Janeiro, Brasil}
\email{simon@puc-rio.br}

\address{Mathematical Institute, University of Oxford, Radcliffe Observatory Quarter, Woodstock Road, Oxford, OX2 6GG, UK}
\email{hurley@maths.ox.ac.uk}

\address{IMPA, Estrada Dona Castorina 110, Jardim Bot\^anico, Rio de Janeiro, 22460-320, Brasil}\email{rob@impa.br}

\address{Department of Pure Mathematics and Mathematical Statistics, Wilberforce Road, Cambridge, CB3 0WA, UK}
\email{jdrs2@cam.ac.uk}

\address{Institute of Mathematics, EPFL SB MATH, MA A2 383 (Bâtiment MA), Station 8, CH-1015 Lausanne, Switzerland}
\email{marius.tiba@epfl.ch}

\thanks{SG was partially supported by FAPERJ (Proc.~201.194/2022) and by CNPq (Proc.~407970/2023-1); EH by the Gravitation Programme NETWORKS (024.002.003) of NWO;  RM by FAPERJ (Proc.~E-26/200.977/2021) and by CNPq (Procs~303681/2020-9 and~407970/2023-1); and JS by European Research Council (ERC) Starting Grant “High Dimensional Probability and Combinatorics”, grant No. 101165900}

\begin{abstract}
The $r$-colour Ramsey number $R_r(k)$ is the minimum $n \in \N$ such that every $r$-colouring of the edges of the complete graph $K_n$ on $n$ vertices contains a monochromatic copy of $K_k$. We prove, for each fixed $r \ge 2$, that 
$$R_r(k) \le e^{-\delta k} r^{rk}$$ 
for some constant $\delta = \delta(r) > 0$ and all sufficiently large $k \in \N$. For each $r \ge 3$, this is the first exponential improvement over the upper bound of~Erd\H{o}s and Szekeres from 1935. In the case $r = 2$, it gives a different (and significantly shorter) proof of a recent result of Campos, Griffiths, Morris and Sahasrabudhe.  
\end{abstract}

\maketitle

\section{Introduction}

One of the most classical (and notorious) problems in combinatorics is to estimate the $r$-colour Ramsey numbers $R_r(k)$, the minimum $n \in \N$ such that every $r$-colouring of the edges of the complete graph on $n$ vertices contains a monochromatic clique on $k$ vertices. Ramsey~\cite{R30} proved in 1930 that $R_r(k)$ is finite for every $k,r \in \N$, and a few years later his theorem was rediscovered in a famous paper of Erd\H{o}s and Szekeres~\cite{ESz35}, whose method gives the bound $R_r(k) \le r^{rk}$. 
This bound turns out to be not too far from the truth: Erd\H{o}s~\cite{E47} showed in 1947 that a random colouring gives an exponential lower bound on $R_2(k)$,  
and his proof can easily be adapted to show that $R_r(k)\ge r^{k/2}$. When $r \ge 3$, this lower bound was significantly improved by Abbott~\cite{A72}, 
who combined Erd\H{o}s' colouring with a simple product construction to show that  
\begin{equation}\label{eq:Rrk:known:bounds}
c^{rk} \le R_r(k) \le r^{rk}
\end{equation}
for some constant $c > 1$. The value of the constant $c$ in~\eqref{eq:Rrk:known:bounds} has been improved in recent years, first by Conlon and Ferber~\cite{CF}, and subsequently by Wigderson~\cite{W} and Sawin~\cite{S}, though in the case $r = 2$ the best known bound (proved in~\cite{S77}) still uses Erd\H{o}s' random colouring, and only improves the bound from~\cite{E47} by a factor of $2$. 

Progress on the upper bound has been much slower, except in the case $r = 2$, where the upper bound of Erd\H{o}s and Szekeres was improved by R\"odl (see~\cite{GR}) and Thomason~\cite{T88} in the 1980s, by Conlon~\cite{C09} in 2009, and by Sah~\cite{S23} in 2023, leading to a bound of the form
$$R_2(k) \le e^{- c(\log k)^2} 4^k$$
for some constant $c > 0$. An exponential improvement was finally obtained in 2023 by Campos, Griffiths, Morris and Sahasrabudhe~\cite{CGMS}, who showed that
$$R_2(k) \le (4 - \eps)^k$$ 
for some constant $\eps > 0$ and all sufficiently large $k \in \N$. The method of~\cite{CGMS} has since been adapted and optimised by Gupta, Ndiaye, Norin and Wei~\cite{GNNW}, who obtained the bound
$$R_2(k) \le 3.8^k.$$ 
The authors of~\cite{GNNW} also noted that the method of~\cite{CGMS} can easily be extended to certain multicolour off-diagonal Ramsey numbers (see below), but when $r \ge 3$ it does not seem to be strong enough to give any improvement at all over the 
bound from~\cite{ESz35} for $r$-colour Ramsey numbers that are close to the diagonal. In fact, we are not aware of \emph{any} previous improvement\footnote{Conlon suggested in his PhD thesis~\cite[page~46]{CPhD} that it may be possible to adapt his method to give a similar improvement for $R_r(k)$ when $r \ge 3$, but also noted several obstructions to obtaining such a result that would seem to require significant additional ideas to overcome.} 
on the upper bound of Erd\H{o}s and Szekeres for 
$R_r(k)$ when $r \ge 3$. 

In this paper we will introduce a new approach to proving upper bounds on $R_r(k)$, and use it to give an exponential improvement over~\cite{ESz35} for all fixed $r \ge 2$. 
 
\begin{theorem}\label{thm:Ramsey:multicolour}
For each $r \ge 2$, there exists $\delta = \delta(r) > 0$ such that 
$$R_r(k) \le e^{-\delta k} r^{rk}$$ 
for all sufficiently large $k \in \N$. 
\end{theorem}

In particular, in the case $r = 2$ we will provide a different (and much shorter) proof of the main result from~\cite{CGMS}. Moreover, the constant $\delta(r)$ given by our proof is polynomial in $r$, so we obtain an improvement over the bound of Erd\H{o}s and Szekeres for all $r = k^{o(1)}$. 

The main new ingredient in our proof of Theorem~\ref{thm:Ramsey:multicolour} 
is a geometric lemma (see Lemma~\ref{lem:lambda}, below) which (roughly speaking) says that if $r$ functions $f_1,\ldots,f_r \colon X \to \R^n$ defined on a finite set $X$ exhibit a large amount of negative correlation, in the sense that 
$$\Pr\Big( \big\< f_i(x), f_i(y) \big\> > - 1 \,\text{ for all }\, i \in [r] \Big) \approx 0,$$
where $x$ and $y$ are independent random elements of $X$, then one of the functions $f_\ell$ must exhibit a significant amount of `clustering', in the sense that it maps many pairs of elements of $X$ to pairs of points with large inner product. The second key new idea is to build $r$ monochromatic books (instead of only one, as in~\cite{CGMS}), using the geometric lemma to significantly boost the density of edges of colour $\ell$ between our reservoir set and the book of colour~$\ell$ whenever performing a `normal' (Erd\H{o}s--Szekeres) step in any of the colours would cause the density of edges in that colour between the reservoir and the corresponding book to decrease too much. As a consequence, we will be able to build a $(t,m)$-book (that is, a clique of size $t$ connected to $m$ extra vertices) with $t = \eps k$ for some small constant $\eps > 0$, and $m \approx n / r^t$ pages, avoiding the delicate tradeoffs and calculations of~\cite{CGMS}. 

The rest of this short paper is organised as follows. In Section~\ref{sec:book:thm} we will state our main technical result about the existence of monochromatic books in $r$-colourings, and describe the algorithm that we will use to prove it. In Section~\ref{sec:key:lemma} we will prove the key lemma, and in Section~\ref{sec:book:proof} we will use it to deduce the book theorem. Finally, in Section~\ref{sec:final:proof}, we will deduce our bound on $R_r(k)$ from the book theorem.

\section{The Book Theorem}\label{sec:book:thm}

The book $(A,B)$ is the graph formed by removing the clique with vertex set $B$ from the clique with vertex set $A \cup B$; that is, it is the graph $H$ with vertex set $A \cup B$ and edge set
$$E(H) = \big\{ uv : \{u,v\} \not\subset B \big\},$$
We say that $H$ is a \emph{$(t,m)$-book} if $|A| = t$ and $|B|= m$, and the sets $A$ and $B$ are disjoint. Our plan is to find a monochromatic copy of $K_k$ (in an arbitrary $r$-colouring $\chi$ of $E(K_n)$) by first finding a monochromatic $(t,m)$-book, where
$$t = \eps k \qquad \text{and} \qquad m \ge e^{-t^2/8k} \, r^{-t} \cdot n \ge R_r(k,\ldots,k,k-t),$$
for some constant $\eps = \eps(r) > 0$, where $R_r(k_1,\ldots,k_r)$ is the minimum $N$ such that every $r$-colouring of $E(K_N)$ contains a monochromatic copy of $K_{k_i}$ in colour $i$ for some $i \in [r]$. Since we have
$$R_r(k,\ldots,k,k-t) \le {kr - t \choose k,\ldots,k,k-t} \le e^{-t^2/6k} r^{rk-t},$$
by the method of Erd\H{o}s and Szekeres~\cite{ESz35}, this will suffice to prove Theorem~\ref{thm:Ramsey:multicolour}. 

We will find a monochromatic $(t,m)$-book by constructing $r$ monochromatic books, one in each colour. The first step is to find a large subset of the vertices in which the colouring is almost regular, and every colour has density close to $1/r$. To do so, we simply run the Erd\H{o}s--Szekeres algorithm until we find such a subset: since we `win' a factor of $1 + \eps$ in each step (see Lemma~\ref{lem:ESz:steps}), either we find a suitable set within $\eps k$ steps, or we are already done. We then (randomly) partition\footnote{In fact, this is unnecessary: we can simply take $X = Y_1 = \cdots = Y_r$ all equal to the set of remaining vertices. However, the reader may find it easier to picture the sets as disjoint.} the remaining vertices into $r+1$ sets: a `reservoir' set $X$, and sets $Y_1,\ldots,Y_r$ that we will use to construct our books. We will find a set $T \subset X$ of size $t = \eps k$ such that $T$ induces a monochromatic clique in some colour $i \in [r]$, and 
$$\bigg| \bigcap_{u \in T} N_i(u) \cap Y_i \bigg| \ge m,$$
where we write $N_i(u)$ to denote the neighbourhood of $u$ in colour $i$. 

We are now ready to state our main technical theorem about finding monochromatic books in $r$-colourings. With an eye to potential future applications, we will state a more general version than we need for Theorem~\ref{thm:Ramsey:multicolour}. In particular, the parameter $\mu$ (which in our application will be $\Theta(r^3)$) allows us to control the loss in the size of the book (relative to $p^t |Y_i|$) in terms of the size of the reservoir set $X$. To avoid repetition, let us fix $n,r \in \N$.

\begin{theorem}\label{thm:book}
Let\/ $\chi$ be an\/ $r$-colouring of\/ $E(K_n)$, and let\/ $X,Y_1,\ldots,Y_r \subset V(K_n)$. 
For every $p > 0$ and $\mu \ge 2^{10} r^3$, and every $t,m \in \N$ with $t \ge \mu^5 / p^2$, the following holds. If 
$$|N_i(x) \cap Y_i| \ge p|Y_i|$$
for every $x \in X$ and $i \in [r]$, and moreover
$$|X| \ge \bigg( \frac{\mu^2}{p} \bigg)^{\mu r t} \qquad \text{and} \qquad |Y_i| \ge \bigg( \frac{e^{2^{13} r^3 / \mu^2}}{p} \bigg)^t \, m,$$ 
then $\chi$ contains a monochromatic $(t,m)$-book.
\end{theorem}


The algorithm that we will use to prove Theorem~\ref{thm:book} is somewhat simpler than the one introduced in~\cite{CGMS}. It is based on the following key lemma, which we will use in each step. Given sets $X,Y \subset V(K_n)$ and a colour $i \in [r]$, define 
$$p_i(X,Y) = \min\bigg\{ \frac{|N_i(x) \cap Y|}{|Y|} : x \in X \bigg\},$$
and note that $p_i(X',Y) \ge p_i(X,Y)$ for every subset $X' \subset X$. Set $\beta = 3^{-4r}$ and $C = 4r^{3/2}$.
 
\begin{lemma}\label{key:lemma}
Let\/ $\chi$ be an\/ $r$-colouring of\/ $E(K_n)$, let\/ $X,Y_1,\ldots,Y_r \subset V(K_n)$ be non-empty sets of vertices, and let $\alpha_1,\ldots,\alpha_r > 0$. There exists a vertex $x \in X$, a colour $\ell \in [r]$, sets $X' \subset X$ and\/ $Y'_1,\ldots,Y'_r\,$ with\/ $Y'_i \subset N_i(x) \cap Y_i\,$ for each $i \in [r]$, and\/ $\lambda \ge -1$, such that 
\begin{equation}\label{eq:key:ell}
|X'| \ge \beta e^{- C \sqrt{\lambda + 1}} |X| \qquad \text{and} \qquad p_\ell( X', Y'_\ell ) \ge p_\ell(X,Y_\ell) + \lambda \alpha_\ell,
\end{equation}
and moreover
\begin{equation}\label{eq:key:alli}
|Y'_i| = p_i(X,Y_i) |Y_i| \qquad \text{and} \qquad p_i( X', Y'_i ) \ge p_i(X,Y_i) - \alpha_i
\end{equation}
for every $i \in [r]$.
\end{lemma}

Since the statement of the lemma might seem a little confusing at first sight, before continuing let us explain roughly how we use it in the algorithm below. We will build a collection of sets $T_1,\ldots,T_r$, with each $T_i$ being a monochromatic clique in colour $i$, and replace the sets $X$ and $Y_1,\ldots,Y_r$ by subsets satisfying
$$X \subset \bigcap_{j \in [r]} \bigcap_{u \in T_j} N_j(u) \qquad \text{and} \qquad Y_i \subset \bigcap_{u \in T_i} N_i(u)$$
for each $i \in [r]$, and such that $p_i(X,Y_i)$ does not decrease too much below its initial value. In each step we will apply Lemma~\ref{key:lemma} and then, depending on the value of $\lambda$, either add the vertex $x$ to one of the sets $T_j$ (a `colour step'), or replace $X$ and $Y_\ell$ by $X'$ and $Y_\ell'$, and observe that $p_\ell(X,Y_\ell)$ increases significantly (a `density-boost step'). 

To be more precise, if the $\lambda$ given by Lemma~\ref{key:lemma} is `small', then we will choose a colour $j \in [r]$ such that the set $X'' = N_j(x) \cap X'$ is as large as possible, and update as follows:
$$X \to X'', \qquad Y_j \to Y'_j \qquad \text{and} \qquad T_j \to T_j \cup \{x\}.$$
This update does not cause $p_j(X,Y_j)$ to decrease too much (by~\eqref{eq:key:alli}), and does not cause $p_i(X,Y_i)$ to decrease at all if $i \ne j$, since the set $Y_i$ does not change and $X'' \subset X$. Moreover, since $\lambda$ is small, the update does not cause the set $X$ to shrink too much. 

If $\lambda$ is `large', on the other hand, then we instead update the sets as follows:
$$X \to X' \qquad \text{and} \qquad Y_\ell \to Y'_\ell,$$
with all other sets staying the same. This time $X$ shrinks by a much larger factor, but to compensate, $p_\ell(X,Y_\ell)$ increases by $\lambda \alpha_\ell$. It will turn out that we are happy with this trade-off because the function $e^{- C \sqrt{\lambda + 1}}$ that appears in our bound~\eqref{eq:key:ell} on the size of $X'$ is sub-exponential in $\lambda$ (see below for discussion of why this is sufficient). 

We are now ready to define the algorithm that we will use to prove Theorem~\ref{thm:book}. The inputs to the algorithm are an $r$-colouring $\chi$ of $E(K_n)$, 
sets $X,Y_1,\ldots, Y_r \subset V(K_n)$, and parameters $t \in \N$ (the size of the book that we are trying to build), $\lambda_0 \ge -1$ (the cut-off between values of $\lambda$ that are `small' and `large'), and $\delta > 0$ (a small constant that bounds the total decrease in $p_i(X,Y_i)$). We also set  
$$p_0 = \min\big\{ p_i(X,Y_i) : i \in [r] \big\},$$
and emphasize that $t$, $\lambda_0$, $\delta$ and $p_0$ remain fixed throughout the algorithm. 

\begin{RBalg}
Set $T_1 = \cdots = T_r = \emptyset$, and repeat the following steps until either $X = \emptyset$ or $\max\big\{ |T_i| : i \in [r] \big\} = t$. 
\begin{enumerate}[label=\arabic*., ref=\arabic*] 
\item\label{Alg:Step1} Applying the key lemma: let the vertex $x \in X$, the colour $\ell \in [r]$, the sets $X' \subset X$ and $Y'_1,\ldots,Y'_r$, and $\lambda \ge -1$ be given by Lemma~\ref{key:lemma}, applied with
\begin{equation}\label{def:alpha}
\alpha_i = \frac{p_i(X,Y_i) - p_0 + \delta}{t}
\end{equation}
for each $i \in [r]$, and go to Step~2.\smallskip
\item\label{Alg:Step2} Colour step: If $\lambda \le \lambda_0$, then choose a colour $j \in [r]$ such that the set
$$X'' = N_j(x) \cap X'$$ 
has at least $(|X'| - 1)/r$ elements, and update the sets as follows:
$$X \to X'', \qquad Y_j \to Y'_j \qquad \text{and} \qquad T_j \to T_j \cup \{x\}$$
and go to Step~1. Otherwise go to Step~3.\smallskip
\item\label{Alg:Step3} Density-boost step: If $\lambda > \lambda_0$, then we update the sets as follows:
$$X \to X' \qquad \text{and} \qquad Y_\ell \to Y'_\ell,$$
and go to Step~1.
\end{enumerate}  
\end{RBalg} 

We remark that in our application of the algorithm, both $\lambda_0$ and $\delta$ will be polynomial functions of $r$, and our sets and colouring will satisfy $p_0 \ge 1/r - \delta$. Let us also emphasize once again that we only update one of the sets $Y_i$ in each round of the algorithm (either the set $Y_j$ in Step~2 or the set $Y_\ell$ in Step~3), and at most one of the sets $T_i$. 

It remains to explain our choice of $\alpha_i$, and why it is important that the bound on the size of $X'$ given by Lemma~\ref{key:lemma} is sub-exponential in $\lambda$. The definition of $\alpha_i$ is just a continuous version of the definition of the parameter $\alpha$ from~\cite{CGMS}, with the parameter $\delta$ controlling the minimum value that it can take. The function~\eqref{def:alpha} has two important properties: it will never be smaller than $\delta / 4t$, and it increases linearly with $p_i(X,Y_i)$, which implies (see Lemma~\ref{lem:Bi:max}) that the total number of density-boost steps is much smaller than $t$, as long as $\lambda_0 \gg \log(1/\delta)$. This second property of $\alpha_i$ moreover allows us to bound the sum of the $\lambda$ that are used in density-boost steps by $O\big( \log(1/\delta) \cdot t \big)$, which in turn implies that our reservoir set $X$ does not shrink too much as a result of density-boost steps (see Lemmas~\ref{lem:X:lower:bound} and~\ref{lem:sum:of:lambdas}). It is here that we require the factor by which $X$ shrinks to be a sub-exponential function of $\lambda$.

\section{The Proof of the Key Lemma}\label{sec:key:lemma}

In this section we will prove Lemma~\ref{key:lemma}. The main step in the proof is the following geometric lemma, which is the most important new ingredient in the proof of Theorem~\ref{thm:Ramsey:multicolour}. The proof of the geometric lemma is (to us, at least) surprisingly short and simple. 

Recall that we fixed $n,r \in \N$, and defined $\beta = 3^{-4r}$ and $C = 4r^{3/2}$. In this section we will write $\langle\cdot,\cdot\rangle$ to denote the standard inner product on $\R^n$.  

\begin{lemma}\label{lem:lambda}
Let\/ $U$ and\/ $U'$ be i.i.d.~random variables taking values in a finite set~$X$, and let $\sigma_1,\ldots,\sigma_r \colon X \to \R^n$ be arbitrary functions. There exist $\lambda \ge -1$ and\/ $i \in [r]$ such that
$$\Pr\Big( \big\langle \sigma_i(U),\sigma_i(U') \big\rangle \ge \lambda \, \text{ and } \, \big\langle \sigma_j(U), \sigma_j(U') \big\rangle \ge -1 \, \text{ for all } \, j \ne i \Big) \ge \beta e^{- C\sqrt{\lambda + 1}}.$$
\end{lemma}

The idea of the proof is as follows. We will first observe (see Lemma~\ref{lem:moments}) that all of the moments of the inner products $\langle \sigma_i(U),\sigma_i(U') \rangle$ are positive. The key step is then to define a function $f \colon \R^r \to \R$ (see~\eqref{def:f}) whose Taylor expansion has non-negative coefficients, that is only positive close to the positive quadrant, and that does not grow too fast. With this function in hand, the lemma will then follow from a simple calculation. 

We begin with the following simple but key observation. 

\begin{lemma}\label{lem:moments}
Let $U$ and\/ $U'$ be i.i.d.~random variables taking values in a finite set~$X$, and let\/ $\sigma_1,\ldots,\sigma_r \colon X \to \R^n$ be arbitrary functions. Then
$$\Ex\Big[ \big\langle \sigma_1(U),\sigma_1(U') \big\rangle^{\ell_1} \cdots \big\langle \sigma_r(U),\sigma_r(U') \big\rangle^{\ell_r} \Big] \ge 0$$
for every $(\ell_1,\dots,\ell_r) \in \Z^r$ with $\ell_1,\dots,\ell_r \ge 0$.
\end{lemma}

\begin{proof}
To simplify the notation, let us write 
$$\big\langle \sigma_1(U),\sigma_1(U') \big\rangle^{\ell_1} \cdots \big\langle \sigma_r(U), \sigma_r(U') \big\rangle^{\ell_r} = \prod_{i = 1}^\ell \big\langle \sigma_{a_i}(U), \sigma_{a_i}(U') \big\rangle $$
where $\ell = \ell_1 + \cdots + \ell_r$ and $(a_1,\dots,a_\ell)$ is such that $\big| \big\{ i \in [\ell] : a_i = j \big\} \big| = \ell_j$ for each $j \in [r]$. Now set
$$Z = \sigma_{a_1}(U) \otimes \sigma_{a_2}(U) \otimes \cdots \otimes \sigma_{a_\ell}(U) \quad \text{and} \quad Z' = \sigma_{a_1}(U') \otimes \sigma_{a_2}(U') \otimes \cdots \otimes \sigma_{a_\ell}(U'),$$ 
and note that 
$$\big\langle Z, Z' \big\rangle = \prod_{i = 1}^\ell \big\langle \sigma_{a_i}(U), \sigma_{a_i}(U') \big\rangle,$$ 
since $\langle u_1 \otimes \cdots \otimes u_r, v_1 \otimes \cdots \otimes v_r \rangle = \langle u_1, v_1 \rangle  \cdots \langle v_r , u_r \rangle $ for any vectors $u_i,v_i \in \R^n$. Finally, note that $Z$ and $Z'$ are independent and identically distributed random vectors, and therefore
$$\Ex \big[ \langle Z, Z' \rangle \big] = \Ex_{Z'} \big[ \Ex_Z \big[ \langle Z, Z' \rangle \big] \big] = 
\Ex_{Z'} \big[ \big\langle \Ex[Z], Z' \big\rangle \big] = \big\langle \Ex[Z], \Ex[Z'] \big\rangle \ge 0,$$ 
as required, where the final inequality holds because $\Ex[Z] = \Ex[Z']$. 
\end{proof}

Next, we will need the following special function: 
\begin{equation}\label{def:f}
f(x_1,\dots,x_r) = \sum_{j = 1}^r x_j \prod_{i \ne j} \big( 2 + \cosh\sqrt{x_i} \big),
\end{equation}
where we define $\cosh \sqrt{x}$ via its Taylor expansion
$$\cosh\sqrt{x} = \sum_{n = 0}^\infty \frac{x^n}{(2n)!}.$$ 
In particular, note that all of the coefficients of the Taylor expansion of $f$ are non-negative. The function $f$ also satisfies the following two inequalities. 

\begin{lemma}\label{lem:special:function}
Let $r \in \N$. The function $f \colon \R^r \to \R$ defined in~\eqref{def:f} satisfies
$$
f(x_1,\dots,x_r) \le \left\{\begin{array}{cl}
3^r r \exp\bigg( \displaystyle\sum_{i = 1}^r \sqrt{ x_i + 3r } \bigg) \quad & \text{if } \,\, x_i \ge - 3r \,\text{ for all }\, i \in [r];\\[+3ex]
-1 & \text{otherwise.} 
\end{array} \right.
$$
\end{lemma}

\begin{proof}
Observe first that $x \le 2 + \cosh\sqrt{x} \le 3e^{\sqrt{x}}$ for every $x \ge 0$, 
that 
$$f(x_1,\dots,x_r) = \bigg( \prod_{i = 1}^r \big( 2 + \cosh \sqrt{x_i} \big) \bigg) \sum_{j = 1}^r \frac{x_j}{2+\cosh\sqrt{x_j}},$$
and that $\cosh \sqrt{x} = \cos \sqrt{-x}$, and hence $-1 \le \cosh \sqrt{x} \le 1$ for all $x < 0$. It follows that
\begin{equation}\label{eq:cosh:bounds}
1 \le 2 + \cosh\sqrt{x} \le 3e^{\sqrt{x+3r}} \qquad \text{and} \qquad \frac{x}{2+\cosh\sqrt{x}} \le 1
\end{equation}
for all $x \ge -3r$, and hence
$$f(x_1,\dots,x_r)\le r \prod_{i=1}^r 3e^{\sqrt{x_i + 3r}} = 3^r r \exp\bigg(\sum_{i=1}^r \sqrt{x_i + 3r} \bigg),$$
for every $x_1,\ldots,x_r \ge -3r$, as claimed. Moreover, if $x_i \le -3r$, then, again using~\eqref{eq:cosh:bounds}, 
$$\sum_{j = 1}^r \frac{x_j}{2+\cosh\sqrt{x_j}} \le \frac{x_i}{3} + r - 1 \le - 1.$$
Since $2 + \cosh \sqrt{x} \ge 1$ for every $x \in \R$, it follows that $f(x_1,\dots,x_r) \le -1$, as required. 
\end{proof}

\pagebreak

\begin{remark}
The key idea in the lemma above was to find an entire function, $2 + \cosh\sqrt{x}$, which $(a)$ has a Taylor expansion at $x = 0$ with non-negative coefficients; $(b)$ is bounded on the negative real axis; and $(c)$ does not grow too quickly on the positive axis. The Phragm\'en--Lindel\"of theorem 
implies that functions satisfying $(a)$ and $(b)$ must grow at least as fast as $\exp\big( x^{1/2+o(1)} \big)$ on the positive real axis.  Thus the bound on the growth of $f$ given by the lemma is essentially best possible for constructions of this type.
\end{remark}

We are now ready to prove our geometric lemma. 

\begin{proof}[Proof of Lemma~\ref{lem:lambda}]
For each $i \in [r]$, define $Z_i = 3r\big\langle \sigma_i(U),\sigma_i(U') \big\rangle$, and observe that, by Lemma~\ref{lem:moments} and linearity of expectation, we have
\begin{equation}\label{eq:Ex:f:positive} 
\Ex\big[ f\big( Z_1,\ldots,Z_r \big) \big] \ge 0,
\end{equation}
since all of the coefficients of the Taylor expansion of $f$ are non-negative. By Lemma~\ref{lem:special:function}, it follows that if we define $E$ to be the event that $Z_i \ge -3r$ 
for every $i \in [r]$, then 
\begin{equation}\label{eq:eventE:inequality} 
3^r r \cdot \Ex\bigg[ \exp\bigg( \sum_{i = 1}^r \sqrt{ Z_i + 3r } \bigg) \mathbf{1}_E \bigg] \ge 1 - \Pr(E),
\end{equation}
where $\mathbf{1}_E \in \{0,1\}$ denotes the indicator of the event $E$, since the left-hand side bounds the 
contribution to the expectation on the left-hand side of~\eqref{eq:Ex:f:positive} due to the event $E$, and the right-hand side bounds the 
contribution to the expectation due to the event $E^c$. 

The result now follows from a simple calculation. First, note that 
$$\Pr(E) = \Pr\Big( Z_i \ge - 3r \, \text{ for all } \, i \in [r] \Big),$$
so if $\Pr(E) \ge \beta$, then the claimed inequality holds with $\lambda = -1$. We claim that if $\Pr(E) \le \beta$, then there exists $\lambda \ge -1$ such that
\begin{equation}\label{eq:max:big:and:E}
 \Pr\Big( \big\{ M \ge \lambda \big\} \cap E \Big) \ge \beta r e^{-C\sqrt{\lambda + 1}},
\end{equation}
where $M = \max \big\{ \big\langle \sigma_i(U),\sigma_i(U') \big\rangle : i \in [r] \big\}$, 
and therefore (by the union bound) there exists an $i \in [r]$ as required. To see this, set $f(x) = e^{c \sqrt{x + 1}}$ with $c = \sqrt{3}r^{3/2}$, and observe that 
$$\Ex\big[ f(M) \mathbf{1}_E \big] \, \le \, \Pr(E) + \int_{-1}^\infty \Pr\Big( \big\{ M \ge \lambda \big\} \cap E \Big) \cdot f'(\lambda) \,\mathrm{d}\lambda$$
since $f(-1) = 1$ and if $E$ holds then $M \ge -1$. It follows that if $\Pr(E) \le \beta$ and~\eqref{eq:max:big:and:E} does not hold for any $\lambda \ge -1$, then 
\begin{align*}
\Ex\Big[ \exp\big( c \sqrt{M + 1} \big) \mathbf{1}_E \Big]
& \le \, \Pr(E) + \int_{-1}^\infty \Pr\Big( \big\{ M \ge \lambda \big\} \cap E \Big) \cdot \frac{c}{2\sqrt{\lambda + 1}} \cdot e^{c \sqrt{\lambda + 1}} \,\mathrm{d}\lambda\\
& \le \, \beta + \beta r \int_{-1}^\infty \frac{c}{2\sqrt{\lambda + 1}} \cdot e^{- \sqrt{\lambda + 1}} \,\mathrm{d}\lambda \, \le \, \beta (cr + 1),
\end{align*}
where 
in the second step we used our choice of $C = 4r^{3/2} \ge c + 1$, 
and in the final step we used the fact that $\int_0^\infty \frac{1}{2\sqrt{x}} e^{-\sqrt{x}} \, \mathrm{d}x = 1$. In particular, 
recalling that 
$\beta = 3^{-4r}$, it follows that the left-hand side of~\eqref{eq:eventE:inequality} is at most $3^r r \cdot 3^{-4r} \big( \sqrt{3}r^{5/2} + 1 \big) \le 1/2$, which contradicts our assumption that $\Pr(E) \le \beta$. Hence there exists $\lambda \ge -1$ such that~\eqref{eq:max:big:and:E} holds, as claimed.
\end{proof}

It is now straightforward to deduce Lemma~\ref{key:lemma} from Lemma~\ref{lem:lambda}. 


\begin{proof}[Proof of Lemma~\ref{key:lemma}] 
For each colour $i \in [r]$, define a function $\sigma_i \colon X \to \R^{Y_i}$ as follows: for each $x \in X$, choose a set $N'_i(x) \subset N_i(x) \cap Y_i$ of size exactly $p_i|Y_i|$, where $p_i = p_i(X,Y_i)$, and set
$$\sigma_i(x) = \frac{\id_{N'_i(x)} - p_i\id_{Y_i}}{\sqrt{\alpha_ip_i|Y_i|}},$$
where $\id_S \in \{0,1\}^{Y_i}$ denotes the indicator function of the set $S$. Note that, for any $x,y\in X$,
$$\big\langle \sigma_i(x),\sigma_i(y) \big\rangle \ge \lambda \quad \Leftrightarrow \quad |N'_i(x) \cap N'_i(y)|\ge \big( p_i + \lambda\alpha_i \big) p_i |Y_i|.$$
Now, by Lemma~\ref{lem:lambda}, there exists $\lambda \ge -1$ and a colour $\ell \in [r]$ such that
$$\Pr\Big( \big\langle \sigma_\ell(U),\sigma_\ell(U') \big\rangle \ge \lambda \, \text{ and } \, \big\langle \sigma_i(U), \sigma_i(U') \big\rangle \ge -1 \, \text{ for all } \, i \ne \ell \Big) \ge \beta e^{- C\sqrt{\lambda + 1}}.$$
where $U$, $U'$ are independent random variables distributed uniformly in the set~$X$. Hence there exists a vertex $x \in X$ and a set $X' \subset X$ such that, 
$$|X'| \ge \beta e^{- C \sqrt{\lambda + 1}} |X| \qquad \text{and} \qquad |N'_\ell(x) \cap N'_\ell(y)| \ge \big( p_\ell + \lambda\alpha_\ell \big) p_\ell |Y_\ell |$$
for every $y \in X'$, and
$$|N'_i(x) \cap N'_i(y)|\ge \big( p_i - \alpha_i \big) p_i |Y_i|$$
for every $y \in X'$ and $i \in [r]$. Setting $Y'_i = N'_i(x)$ for each $i \in [r]$, it follows that
$$p_\ell\big( X', Y'_\ell \big) \ge p_\ell(X,Y_\ell) + \lambda \alpha_\ell \qquad \text{and} \qquad p_i\big( X', Y'_i \big) \ge p_i(X,Y_i) - \alpha_i$$
for every $i \in [r]$, as required. 
\end{proof}

\section{Proof of the Book Theorem}\label{sec:book:proof}

In this section we will use the multicolour book algorithm to prove Theorem~\ref{thm:book}. To do so, we will first prove a few simple lemmas about the sets produced by the algorithm when applied with arbitrary inputs, and then apply these bounds to our setting. Recall that we are given an $r$-colouring $\chi$ of $E(K_n)$ and 
sets $X,Y_1,\ldots,Y_r \subset V(K_n)$. Define 
\begin{equation}\label{def:p0}
p_0 = \min\big\{ p_i(X,Y_i) : i \in [r] \big\},
\end{equation}
and run the algorithm for some $t \in \N$, $\lambda_0 \ge 0$ and $\delta > 0$.

In order to simplify the statements of the lemmas below, let us write $X(s)$ and $Y_i(s)$ for the sets $X$ and $Y_i$ after $s$ steps of the algorithm, and set 
$$p_i(s) = p_i\big( X(s), Y_i(s) \big) \qquad \text{and} \qquad \alpha_i(s) = \frac{p_i(s) - p_0 + \delta}{t}.$$ 
Let us also write $\ell(s) \in [r]$ and $\lambda(s) \ge -1$ for the colour and number given by the application of Lemma~\ref{key:lemma} to the sets $X(s)$ and $Y_1(s),\dots,Y_r(s)$, and define
$$\cB_i(s) = \big\{ 0 \le j < s \,:\, \ell(j) = i \, \text{ and } \, \lambda(j) > \lambda_0 \big\},$$
for each $i \in [r]$ and $s \in \N$, to be the set of density-boost steps in colour $i$ during the first $s$ steps of the algorithm. We begin by noting the following lower bound on $p_i(s)$. 

\begin{lemma}\label{lem:pi:lower:bound}
For each $i \in [r]$ and $s \in \N$, 
\begin{equation}\label{eq:pi:lower:bound}
p_i(s) - p_0 + \delta \, \ge \, \delta \cdot \bigg( 1 - \frac{1}{t} \bigg)^{t} \prod_{j \in \cB_i(s)} \bigg( 1 + \frac{\lambda(j)}{t} \bigg).
\end{equation}
\end{lemma}

\begin{proof}
Note first that if $Y_i(s+1) = Y_i(s)$, then $p_i(s+1) \ge p_i(s)$, since the minimum degree does not decrease when we take a subset of $X(s)$. When we perform a colour step in colour $i$, 
we have $p_i(s+1) \ge p_i(s) - \alpha_i(s)$, by Lemma~\ref{key:lemma}, and hence
$$p_i(s+1) - p_0 + \delta \ge \bigg( 1 - \frac{1}{t} \bigg) \big( p_i(s) - p_0 + \delta \big),$$
by our choice of $\alpha_i(s)$. Similarly, when we perform a density-boost step in colour $i$  
we have $p_i(s+1) \ge p_i(s) + \lambda(s) \alpha_i(s)$, by Lemma~\ref{key:lemma}, and hence
$$p_i(s+1) - p_0 + \delta \ge \bigg( 1 + \frac{\lambda(s)}{t} \bigg) \big( p_i(s) - p_0 + \delta \big).$$
Recalling that there are at most $t$ colour steps in colour $i$, and that $p_i(0) \ge p_0$, by the definition~\eqref{def:p0} of $p_0$, the claimed bound follows. 
\end{proof}

Before continuing, let's note a couple of important consequences of Lemma~\ref{lem:pi:lower:bound}. First, it implies that neither $p_i(s)$ nor $\alpha_i(s)$ can get too small. 

\begin{lemma}\label{lem:pi:min}
If\/ $t \ge 2$, then 
$$p_i(s) \, \ge \, p_0 - \frac{3\delta}{4} \qquad \text{and} \qquad \alpha_i(s) \, \ge \, \frac{\delta}{4t}$$
for every $i \in [r]$ and $s \in \N$. 
\end{lemma}

\begin{proof}
Both bounds follow immediately from~\eqref{eq:pi:lower:bound} and the definition of $\alpha_i(s)$, since we chose $\lambda_0 \ge 0$ and using the fact that $( 1 - 1/t )^t \ge 1/4$ for every $t \ge 2$. 
\end{proof}

It also implies the following bound on the number of density-boost steps. 

\begin{lemma}\label{lem:Bi:max}
If\/ $t \ge \lambda_0 \ge 2$ and\/ $\delta \le 1/4$, then
$$|\cB_i(s)| \le \frac{4 \log(1/\delta)}{\lambda_0} \cdot t$$
for every $i \in [r]$ and $s \in \N$. 
\end{lemma}

\begin{proof}
Since $\lambda(j) > \lambda_0$ for every $j \in \cB_i(s)$, and $p_i(s) \le 1$, it follows from~\eqref{eq:pi:lower:bound} that
$$\frac{\delta}{4} \bigg( 1 + \frac{\lambda_0}{t} \bigg)^{|\cB_i(s)|} \le 1 + \delta.$$
Since $t \ge \lambda_0$ and $\delta \le 1/4$, the claimed bound follows. 
\end{proof}

Lemmas~\ref{lem:pi:min} and~\ref{lem:Bi:max} together provide a lower bound on the size of the set $Y_i(s)$. 

\begin{lemma}\label{lem:Y:lower:bound}
If\/ $t \ge 2$, then
$$|Y_i(s)| \ge \bigg( p_0 - \frac{3\delta}{4} \bigg)^{t + |\cB_i(s)|} |Y_i(0)|$$
for every $i \in [r]$ and $s \in \N$. 
\end{lemma}

\begin{proof}
Note that $Y_i(j+1) \ne Y_i(j)$ for at most $t + |\cB_i(s)|$ of the first $s$ steps, and for those steps we have
$$|Y_i(j+1)| = p_i(j) |Y_i(j)| \ge \bigg( p_0 - \frac{3\delta}{4} \bigg) |Y_i(j)|,$$ 
by~\eqref{eq:key:alli} and Lemma~\ref{lem:pi:min}.
\end{proof}

Finally, we need to bound the size of the set $X(s)$. To do so, set $\eps = (\beta / r) e^{- C \sqrt{\lambda_0 + 1}}$, and define $\cB(s) = \cB_1(s) \cup \cdots \cup \cB_r(s)$ to be the set of all density-boost steps. 

\begin{lemma}\label{lem:X:lower:bound}
For each $s \in \N$, 
\begin{equation}\label{eq:X:lower:bound}
|X(s)| \ge \eps^{rt + |\cB(s)|} \exp\bigg( - C \sum_{j \in \cB(s)} \sqrt{\lambda(j)+1}\,\, \bigg) |X(0)| - rt.
\end{equation}
\end{lemma}

\begin{proof}
If $\lambda(j) \le \lambda_0$, then by~\eqref{eq:key:ell} and Step~2 of the algorithm we have
$$|X(j+1)| \ge \frac{\beta e^{- C \sqrt{\lambda_0 + 1}}}{r} \cdot |X(j)| - 1 = \eps |X(j)| - 1.$$ 
On the other hand, if $\lambda(j) > \lambda_0$, then $j \in \cB(s)$, and we have 
$$|X(j+1)| \ge \beta e^{- C \sqrt{\lambda(j) + 1}} |X(j)|,$$ 
by~\eqref{eq:key:ell} and Step~3 of the algorithm. Since there are at most $rt$ colour steps, and $\beta \ge \eps$, the claimed bound follows.
\end{proof}

We will use the following lemma to bound the right-hand side of~\eqref{eq:X:lower:bound}. 

\begin{lemma}\label{lem:sum:of:lambdas}
If\/ $t \ge \lambda_0 / \delta > 0$ and\/ $\delta \le 1/4$, then
$$\sum_{j \in \cB(s)} \sqrt{\lambda(j)} \le \frac{7r \log(1/\delta)}{\sqrt{\lambda_0}} \cdot t$$
for every $s \in \N$. 
\end{lemma}

\begin{proof}
Observe first that, by Lemma~\ref{lem:pi:lower:bound}, we have 
\begin{equation}\label{eq:B:condition}
\frac{\delta}{4} \prod_{j \in \cB_i(s)} \bigg( 1 + \frac{\lambda(j)}{t} \bigg) \le 1 + \delta
\end{equation}
for each $i \in [r]$. Note also that $\lambda_0 \le \lambda(j) \le 5t/\delta$ for every $j \in \cB(s)$, where the lower bound holds by the definition of $\cB(s)$, and the upper bound by~\eqref{eq:B:condition} and since $\delta \le 1/4$. Now, since $\log(1+x) \ge \min\{ x/2,1\}$ for all $x > 0$, it follows that
\begin{equation}\label{eq:not:convexity}
\frac{ \sqrt{\lambda(j)} }{\log(1 + \lambda(j)/t) } \le \max\bigg\{ \frac{ 2t }{ \sqrt{\lambda_0} }, \, \sqrt{ \frac{5t}{\delta}} \bigg\} \le \frac{ 3t }{ \sqrt{\lambda_0} },
\end{equation}
where in the second inequality we used 
our assumption that $t \ge \lambda_0 / \delta$. It now follows that
$$\sum_{j \in \cB_i(s)} \sqrt{\lambda(j)} \le \frac{ 3t }{ \sqrt{\lambda_0} } \sum_{j \in \cB_i(s)} \log \bigg( 1 + \frac{\lambda(j)}{t} \bigg) \le \frac{7 \log(1/\delta)}{\sqrt{\lambda_0}} \cdot t,$$
where the first inequality holds by~\eqref{eq:not:convexity}, and the second follows from~\eqref{eq:B:condition}, since $\delta \le 1/4$. Summing over $i \in [r]$, we obtain the claimed bound. 
\end{proof}

We are finally ready to prove the book theorem. 

\begin{proof}[Proof of Theorem~\ref{thm:book}]
Recall that we are given an $r$-colouring $\chi$ of $E(K_n)$ and a collection of 
sets $X,Y_1,\ldots,Y_r \subset V(K_n)$ with
\begin{equation}\label{eq:book:thm:conditions}
|X| \ge \bigg( \frac{\mu^2}{p} \bigg)^{\mu r t} \qquad \text{and} \qquad |Y_i| \ge \bigg( \frac{e^{2^{13} r^3 / \mu^2}}{p} \bigg)^t \, m
\end{equation}
for each $i \in [r]$, for some $p \in (0,1]$, $\mu \ge 2^{10} r^3$ and $t,m \in \N$ with $t \ge \mu^5 / p^2$, and moreover
\begin{equation}\label{eq:book:thm:min:degree}
|N_i(x) \cap Y_i| \ge p|Y_i|
\end{equation}
for every $x \in X$ and $i \in [r]$. 
We will run the multicolour book algorithm with
\begin{equation}\label{def:delta:lambda0}
\delta = \frac{p}{\mu^2} \qquad \text{and} \qquad \lambda_0 = \bigg( \frac{\mu \log(1/\delta)}{8C} \bigg)^2,
\end{equation}
where $C = 4r^{3/2}$ (as before), and show that it ends with
$$\max \big\{ |T_i(s)| : i \in [r] \big\} = t \qquad \text{and} \qquad \min\big\{ |Y_i(s)| : i \in [r] \big\} \ge m,$$
and therefore produces a monochromatic $(t,m)$-book.

To do so, we just need to bound the sizes of the sets $X(s)$ and $Y_i(s)$ from below. Observe that $t \ge \lambda_0$ and $\delta \le 1/4$, and therefore, by Lemma~\ref{lem:Bi:max}, that
\begin{equation}\label{eq:Bi:final:bound}
|\cB_i(s)| \, \le \, \frac{4 \log(1/\delta)}{\lambda_0} \cdot t \, = \, \frac{2^{12} r^3}{\mu^2 \log(1/\delta)} \cdot t \, \le \, t
\end{equation}
for every $i \in [r]$ and $s \in \N$. Since $p_0 \ge p$, by~\eqref{eq:book:thm:min:degree} and the definition of $p_0$, it follows that  
$$|Y_i(s)| \, \ge \, \bigg( p - \frac{3\delta}{4} \bigg)^{t + |\cB_i(s)|} |Y_i| \, \ge \, e^{-2\delta t / p} \, p^{|\cB_i(s)|} \, \big( e^{2^{13} r^3 / \mu^2} \big)^t \, m,$$
where the first inequality holds by Lemma~\ref{lem:Y:lower:bound}, and the second by~\eqref{eq:book:thm:conditions} and~\eqref{eq:Bi:final:bound}. Noting that $p^{|\cB_i(s)|} \ge e^{-2^{12} r^3 t / \mu^2}$, by~\eqref{eq:Bi:final:bound}, and that $\delta / p = 1/\mu^2$, it follows that 
$$|Y_i(s)| \, \ge \, e^{-2 t / \mu^2} \big( e^{2^{12} r^3 / \mu^2} \big)^t \, m \, \ge \, m,$$
as claimed. To bound $|X(s)|$, recall~\eqref{eq:X:lower:bound} and observe that 
$$\eps^{rt + |\cB(s)|}  \ge \bigg( \frac{\beta}{r} \cdot e^{- C \sqrt{\lambda_0 + 1}} \bigg)^{2rt} \ge \big( e^{- 4C \sqrt{\lambda_0}} \big)^{rt} = \delta^{\mu r t/2} = \bigg( \frac{\mu^2}{p} \bigg)^{-\mu r t / 2},$$
where the first step holds because $\eps = (\beta / r) e^{- C \sqrt{\lambda_0 + 1}}$ and $|\cB(s)| \le rt$, by~\eqref{eq:Bi:final:bound}, the second step holds because $\beta = 3^{-4r}$ and $\lambda_0 \ge 2^{10} r^3$, and the third and fourth steps follow from our choice of $\lambda_0$ and $\delta$ in~\eqref{def:delta:lambda0}. 
Moreover, observe that $t \ge \lambda_0 / \delta$, and therefore
$$\sum_{j \in \cB(s)} \sqrt{\lambda(j)+1} \le \frac{8r \log(1/\delta)}{\sqrt{\lambda_0}} \cdot t = \frac{2^6Crt}{\mu},$$
by Lemma~\ref{lem:sum:of:lambdas} and our choice of $\lambda_0$, and therefore 
$$|X(s)| \ge \bigg( \frac{\mu^2}{p} \bigg)^{\mu r t / 2} \exp\bigg( - \frac{2^6C^2rt}{\mu} \bigg) - rt > 0,$$
for every $s \in \N$, by Lemma~\ref{lem:X:lower:bound} and since $\mu \ge 2^{10} r^3$ and $C = 4r^{3/2}$. It follows that the algorithm produces a monochromatic $(t,m)$-book, as claimed. 
\end{proof}

\section{Deducing the bound on $R_r(k)$}\label{sec:final:proof}

In this section we will deduce Theorem~\ref{thm:Ramsey:multicolour}, our bound on the Ramsey numbers $R_r(k)$, from Theorem~\ref{thm:book}. We will in fact prove the following quantitative version of the theorem.

\begin{theorem}\label{thm:Ramsey:multicolour:quant}
Let $r \ge 2$, and set $\delta = 2^{-160} r^{-12}$. Then
$$R_r(k) \le e^{-\delta k} r^{rk}$$ 
for every $k \in \N$ with $k \ge 2^{200} r^{20}$. 
\end{theorem}

In order to apply Theorem~\ref{thm:book}, we need to find a suitable collection of sets $X,Y_1,\ldots,Y_r$ in an arbitrary $r$-colouring of $E(K_n)$. To do so, we simply run the Erd\H{o}s--Szekeres process until we find a subset of the vertex set in which each colour has density close to $1/r$, and the colouring is close to regular in each colour. 

\begin{lemma}\label{lem:ESz:steps}
Given\/ $n,r \in \N$ and\/ $\eps > 0$, and an $r$-colouring of $E(K_n)$, the following holds. There exist disjoint sets of vertices\/ $S_1,\dots,S_r$ and\/ $W$ such that, for every $i \in [r]$, 
$$|W| \ge \bigg( \frac{1+\eps}{r} \bigg)^{\sum_{j = 1}^r |S_j|} n \qquad \text{and} \qquad |N_i(w) \cap W| \ge \bigg( \frac{1}{r} - \eps \bigg) |W| - 1$$ 
for every $w \in W$, and $(S_i,W)$ is a monochromatic book in colour~$i$.
\end{lemma}

\begin{proof}
We use induction on~$n$. Note that if $n \le r$ or $r = 1$, then the sets $S_1 = \cdots = S_r = \emptyset$ and $W = V(K_n)$ have the required properties. We may therefore assume that $n > r \ge 2$, and that there exists a vertex $x \in V(K_n)$ and a colour $\ell \in [r]$ such that
$$|N_\ell(x)| < \bigg( \frac{1}{r} - \eps \bigg) n - 1,$$
and hence there exists a different colour $j \ne \ell$ such that 
$$|N_j(x)| \ge \frac{1}{r-1} \bigg( n - \bigg( \frac{1}{r} - \eps \bigg) n \bigg) \ge \bigg( \frac{1 + \eps}{r} \bigg) n.$$
Applying the induction hypothesis to the colouring induced by the set $N_j(x)$, we obtain sets $S'_1,\ldots,S'_r$ and $W'$ satisfying the conclusion of the lemma with $n$ replaced by $|N_j(x)|$. Setting 
$$W = W', \qquad S_j = S'_j \cup \{x\} \qquad \text{and} \qquad S_i = S'_i \qquad \text{for each } \, i \ne j,$$ 
we see that $W$ satisfies the minimum degree condition (by the induction hypothesis), and
$$|W| \ge \bigg( \frac{1+\eps}{r} \bigg)^{\sum_{i = 1}^r |S'_i|}|N_j(x)| \ge \bigg( \frac{1+\eps}{r} \bigg)^{\sum_{i=1}^r |S_i|}n.$$
Since $(S_i,W)$ is a monochromatic book in colour~$i$ for each $i \in [r]$, it follows that $S_1,\dots,S_r$ and $W$ are sets with the claimed properties.
\end{proof}

The next lemma, which is an immediate consequence of Erd\H{o}s and Szekeres' bound on the $r$-colour Ramsey numbers, implies (roughly speaking) that either the sets $S_1,\dots,S_r$ given by Lemma~\ref{lem:ESz:steps} satisfy $|S_i| = o(k)$ for each $i \in [r]$, or we are already done.

\begin{lemma}\label{lem:many:ESz:steps}
If\/ $k,r \ge 2$ and\/ $\eps \in (0,1)$, then 
$$R_r\big( k - s_1, \ldots, k - s_r \big) \le e^{-\eps^3 k / 2} \bigg( \frac{1+\eps}{r} \bigg)^s r^{rk}$$
for every $s_1,\ldots,s_r \in [k]$ with\/ $s = s_1 + \cdots + s_r \ge \eps^2 k$.
\end{lemma}

\begin{proof}
By the Erd\H{o}s--Szekeres bound, we have
$$R_r\big( k - s_1, \ldots, k - s_r \big) \le r^{rk - s}.$$
The claimed bound follows, since $( 1 + \eps )^{\eps^2 k} \ge e^{\eps^3 k/2}$ for every $\eps \in (0,1)$.
\end{proof}

Our bound on $R_r(k)$ now follows from a simple calculation. 

\begin{proof}[Proof of Theorem~\ref{thm:Ramsey:multicolour:quant}]
Given $r \ge 2$, set $\delta = 2^{-160} r^{-12}$, let $k \ge 2^{200} r^{20}$ and $n \ge e^{-\delta k} r^{rk}$, and let $\chi$ be an arbitrary $r$-colouring of $E(K_n)$. Set $\eps = 2^{-50} r^{-4}$, and let $S_1,\ldots,S_r$ and $W$ be the sets given by Lemma~\ref{lem:ESz:steps}. Suppose first that $|S_1| + \cdots + |S_r| \ge \eps^2 k$, and observe that, by Lemma~\ref{lem:many:ESz:steps} and our lower bounds on $n$ and $|W|$, we have 
$$|W| \ge \bigg( \frac{1+\eps}{r} \bigg)^{\sum_{j = 1}^r |S_j|} e^{-\delta k} r^{rk} \ge R_r\big( k - |S_1|, \ldots, k - |S_r| \big).$$ 
Since $(S_i,W)$ is a monochromatic book in colour~$i$ for each $i \in [r]$, it follows that there exists a monochromatic copy of $K_k$ in $\chi$, as required.

We may therefore assume from now on that  $|S_1| + \cdots + |S_r| \le \eps^2 k$. In this case we set $X = Y_1 = \cdots = Y_r = W$, and apply Theorem~\ref{thm:book} to the colouring restricted to $W$ with 
$$p = \frac{1}{r} - 2\eps, \qquad \mu = 2^{30} r^3, \qquad t = 2^{-40} r^{-3} k \qquad \text{and} \qquad m = R\big( k, \ldots, k, k - t \big).$$
In order to apply Theorem~\ref{thm:book}, we need to check that $t \ge \mu^5 / p^2$, and that
$$|X| \ge \bigg( \frac{\mu^2}{p} \bigg)^{\mu r t} \qquad \text{and} \qquad |Y_i| \ge \bigg( \frac{e^{2^{13} r^3 / \mu^2}}{p} \bigg)^t \, m.$$
The first two of these inequalities are both easy: the first holds because $k \ge 2^{200} r^{20}$, and for the second note that since $n \ge r^{rk/2}$ and $\sum_{j = 1}^r |S_j| \le rk/4$, we have
$$|X| \ge \bigg( \frac{1+\eps}{r} \bigg)^{rk/4} r^{rk/2} \ge r^{rk/4} \ge \big( 2^{61} r^7 \big)^{2^{-10} r k} \ge \bigg( \frac{\mu^2}{p} \bigg)^{\mu r t},$$
as required. The third inequality is more delicate: observe first\footnote{See Appendix~\ref{app:binomial} for a proof of the second inequality.} that 
$$R\big( k, \ldots, k, k-t \big) \le \binom{rk-t}{k,\dots,k,k-t} \le e^{-(r-1)t^2/3rk} r^{rk-t} \le e^{-t^2/6k} r^{rk-t}.$$
and therefore, since $\delta \le 2^{-10} t^2/k^2$ and $p = 1/r - 2\eps \ge e^{-3\eps r} / r$, we have 
$$n \ge e^{-\delta k} r^{rk} \ge r^t e^{t^2/8k} \cdot R\big( k, \ldots, k, k-t \big) \ge \frac{m}{p^t} \cdot \exp\bigg( \frac{t^2}{8k} - 3\eps r t \bigg).$$
Moreover, by our choice of $t$ and $\mu$, we have 
$$\frac{t}{8k} = \frac{1}{2^{43} r^3} \ge \frac{1}{2^{47}r^3} + \frac{1}{2^{48}r^3} = \frac{2^{13} r^3}{\mu^2} + 4\eps r.$$
Finally, since $\sum_{j = 1}^r |S_j| \le \eps^2 k \le \eps t$, it follows that
$$|Y_i| = |W| \ge \bigg( \frac{1+\eps}{r} \bigg)^{\eps t} n \ge \frac{m}{p^t} \cdot \exp\bigg( \frac{t^2}{8k} - 4\eps r t \bigg) \ge \bigg( \frac{e^{2^{13} r^3 / \mu^2}}{p} \bigg)^t \, m,$$
as claimed. Thus $\chi$ contains a monochromatic $(t,m)$-book, and hence, by our choice of $m$, $\chi$ contains a monochromatic copy of $K_k$, as required.
\end{proof}

\section*{Acknowledgements}

Important steps in the work described in this paper were carried out during visits by various subsets of the authors to the University of Cambridge, IMPA, PUC-Rio, the IAS, 
and Sapienza Università di Roma. We thank each of these institutions for providing a wonderful working environment, and also Yuval Wigderson for pointing us to the paper~\cite{A72}. 

\appendix

\section{Bounding binomial coefficients}\label{app:binomial}

In this short appendix we will prove the following bound, which was used in Section~\ref{sec:final:proof}. 

\begin{lemma}\label{lem:multibounds}
For each $k,t,r \in \N$ with $3 \le t \le k$, 
$$\binom{rk-t}{k,\dots,k,k-t} \le e^{-(r-1)t^2/3rk} r^{rk-t}.$$
\end{lemma}
 
\begin{proof}
To prove this, observe that 
$$\binom{rk-t}{k,\dots,k,k-t} \binom{rk}{k,\dots,k}^{-1} = \, \prod_{i = 0}^{t - 1} \frac{k - i}{rk - i} = r^{-t} \,\prod_{i = 0}^{t-1} \bigg( 1 - \frac{(r-1)i}{rk - i} \bigg) \le r^{-t} \cdot e^{-(r-1)t^2/3rk}.$$
Since $\binom{rk}{k,\dots,k} \le r^{rk}$, the claimed inequality follows.
\end{proof}

\end{document}